\documentclass[12pt]{article}

\usepackage{graphicx}
\usepackage{amsmath, amsthm, amssymb, mathtools}
\usepackage[english]{babel}
\usepackage{aliascnt}
\usepackage{tikz}
\usetikzlibrary{matrix,arrows.meta}
\usetikzlibrary{decorations.pathmorphing,shapes}
\usepackage[hidelinks]{hyperref}
\usepackage[capitalize]{cleveref}
\usepackage{multirow}
\usepackage[margin=1in]{geometry}

\makeatletter
\g@addto@macro\@floatboxreset\centering
\makeatother

\def\R{\mathbb{R}}

\def\N{\mathbb{N}}
\renewcommand{\subset}{\subseteq}

\def\empty{\varnothing}
\newcommand{\set}[1]{\{1,\dots,#1\}}
\newcommand{\Set}{\mathbf{Set}}
\newcommand{\Prob}{\mathrm{Prob}}
\newcommand{\Meas}{\mathrm{Meas}}
\newcommand{\FinMeas}{\mathrm{FinMeas}}
\newcommand{\cF}{\mathcal{F}}
\newcommand{\cS}{\mathcal{S}}
\newcommand{\op}{\mathrm{op}}

\crefname{equation}{Equation}{Equations}

\newtheorem{theorem}{Theorem}
\crefname{theorem}{Theorem}{Theorems}

\crefname{lemma}{Lemma}{Lemmas}

\crefname{proposition}{Proposition}{Propositions}

\newtheorem{corollary}[theorem]{Corollary}
\crefname{corollary}{Corollary}{Corollaries}

\crefname{conjecture}{Conjecture}{Conjectures}

\theoremstyle{definition} 

\newtheorem{definition}[theorem]{Definition}
\crefname{definition}{Definition}{Definitions}

\newtheorem{remark}[theorem]{Remark}
\crefname{remark}{Remark}{Remarks}

\newtheorem{example}[theorem]{Example}
\crefname{example}{Example}{Examples}

\crefname{figure}{Figure}{Figures}

\title{Sheaves of Probability}
\author{Owen D.\ Biesel}

\begin{document}

\maketitle

\abstract{What does it mean for multiple agents' credence functions to be consistent with each other, if the agents have distinct but overlapping sets of evidence? Mathematical philosopher Michael Titelbaum's rule, called Generalized Conditionalization (GC), sensibly requires each pair of agents to acquire identical credences if they updated on each other's evidence. However, GC allows for paradoxical arrangements of agent credences that we would not like to call consistent. We interpret GC as a gluing condition in the context of sheaf theory, and show that if we further assume that the agents' evidence is logically consistent then the sheaf condition is satisfied and the paradoxes are resolved.}

\section{Introduction}

Suppose that an agent models the world as a state space $X$ with a credence function given by a probability measure $P$ on $X$. Then after learning evidence $A$ (regarded as a subset of $X$ consisting of those states in which the agent encounters that evidence), the agent has updated credences given by a new probability measure $P'$. For the agent's before-and-after credences to be regarded as consistent, the probability measures $P$ and $P'$ should be related by conditionalization: for each event $B\subset X$, we should have
\[P'(B) = P(B|A) = \frac{P(B\cap A)}{P(A)}.\]
An equivalent viewpoint is that the original probability measure $P$ on $X$ has been \emph{restricted} to a probability measure $P|_A$ on $A$ given by
\[P|_A(B) \coloneqq \frac{P(B)}{P(A)}\text{ for }B\subset A.\]

In \cite{Titelbaum2013}, Michael Titelbaum tackles the question of what it means for what a collection of agents to have consistent credence functions despite having learned arbitrarily overlapping sets of evidence. He uses a criterion he calls \emph{generalized conditionalization}, translated below to the language of restricting probabilities:

\begin{definition}\label{gc}
 Suppose we have agents $1$ through $n$ whose models of the world use a shared state space $X$, and that for each $i\in\set{n}$ agent $i$ has learned evidence $A_i\subset X$ and has a credence function given by a probability measure $P_i$ on $A_i$. The agents' beliefs $P_i$ satisfy \emph{generalized conditionalization} (GC) if for each pair of agents $i$ and $j$, the restrictions $P_i|_{A_i\cap A_j}$ and $P_j|_{A_i\cap A_j}$ agree. 
\end{definition}

\begin{example}\label{not_gc}
 Suppose that agents 1, 2, and 3 have credences given by the following probability mass functions on subsets of state space $X = \{a,b,c,d,e\}$:
 \[\begin{array}{r|c|c|c|c|c}
 & a & b & c & d & e \\ \hline 
 P_1 & 75\% & 10\% & 15\%  & & \\
 P_2 &  & 20\% &  30\% & 50\% & \\
 P_3 & & & 10\% & 40\% & 50\%
 \end{array}\]
 In this setting, agent 1 has learned evidence $A_1 = \{a,b,c\}$, agent 2 has learned $A_2 = \{b,c,d\}$, and agent 3 has learned $A_3 = \{c,d,e\}$. We can check whether the three agents satisfy GC by comparing the restrictions for each pair of agents: 1 and 2, 1 and 3, and 2 and 3.
 
 For agents 1 and 2, we restrict $P_1$ and $P_2$ to the intersection of their evidence $A_1\cap A_2 = \{b,c\}$. Agent 1 assigns total probability $10\% + 15\% = 25\%$ to $\{b,c\}$, while agent 2 assigns probability $20\% + 30\% = 50\%$, so their restricted probability measures are
  \[\renewcommand{\arraystretch}{2}\begin{array}{r|c|c}
 & b &c \\ \hline
 P_1|_{A_1\cap A_2} & \dfrac{10\%}{25\%} = 40\% & \dfrac{15\%}{25\%} = 60\% \\
 P_2|_{A_1\cap A_2} & \dfrac{20\%}{50\%} = 40\% & \dfrac{30\%}{50\%} = 60\%
 \end{array}\]
 These are the same restricted probability measures, so agents 1 and 2 have compatible credences. 
 
 Agents 1 and 3 also have compatible credences: we have $A_1\cap A_3 = \{c\}$ and both $P_1$ and $P_3$ restrict to the probability mass function $c\mapsto 100\%$. However, agents 2 and 3 do not have compatible credences: $P_2$ restricted to $A_2\cap A_3 = \{c,d\}$ assigns probability $30\%/(30\% + 50\%) = 37.5\%$ to $c$, but $P_3$ restricted to $\{c,d\}$ assigns only $10\%/(10\% + 40\%) = 20\%$ to $c$.
 
 Therefore, agents 1 and 2 satisfy GC (as do agents 1 and 3), but agents 1, 2, and 3 all together do not satisfy GC.
\end{example}

Mathematicians familiar with sheaf theory may recognize \cref{gc} as reminiscent of the ``gluing condition'' in the definition of a sheaf (which we review in \cref{sheaves}). The natural question, then, is whether probability measures on $A_1,\ldots,A_n$ satisfying GC ``glue'' to a single probability measure on $\bigcup_{i=1}^n A_i$ that restricts to all of them? In other words, if the agents' beliefs are ``consistent'' in the sense of GC, then could they be considered to have shared a common probability distribution prior to learning their individual sets of evidence?

In \cref{not_gc}, because agents 1 and 2 satisfy GC we can ask whether there is a single probability measure $P$ on $A_1\cup A_2 = \{a,b,c,d\}$ that restricts to $P_1$ on $\{a,b,c\}$ and to $P_2$ on $\{b,c,d\}$. Indeed there is:
\[\renewcommand{\arraystretch}{1.5}\begin{array}{c|c|c|c|c}
 & a & b & c & d  \\ \hline 
P &  60\% &  8\% &  12\%  & 20\% \\
P|_{a,b,c} = P_1 & \dfrac{60\%}{80\%} = 75\% & \dfrac{8\%}{80\%} = 10\% & \dfrac{12\%}{80\%} = 15\% & \vphantom{\sum_\sum}\\
P|_{b,c,d} = P_2 & & \dfrac{8\%}{40\%} = 20\% & \dfrac{12\%}{40\%} = 30\% & \dfrac{20\%}{40\%} = 50\%
 \end{array}\]

However, the probabilities of agents satisfying GC don't always glue together:

\begin{example}
 Consider three agents with the following probability measures on subsets of $X = \{a,b,c\}$:
   \[\begin{array}{r|c|c|c}
 & a &  b &c \\ \hline
 P_1 & 40\% & 60\% & \\
 P_2 &  & 40\% & 60\% \\
 P_3 & 60\% & & 40\%
 \end{array}\]
 These agents satisfy GC: each pair of probability measures restricts to 100\% on the one element of $X$ they have in common. However, there are two senses in which these agents' beliefs are \emph{in}consistent:
 \begin{enumerate}
  \item Logical inconsistency: Collectively, the individual probability measures rule out the entire state space: $a$ is ruled out by $P_2$, $b$ by $P_3$, and $c$ by $P_1$.
  \item Prior inconsistency: There is no single probability measure on all of $X$ that conditionalizes to all three of the $P_i$; we would need to have $P(a)<P(b)<P(c)$ to restrict to $P_1$ and $P_2$, but $P(a)>P(c)$ to restrict to $P_3$.
 \end{enumerate}
\end{example}

The main result of this paper is that these three notions of (in)consistency are related: if a set of agents satisfy both GC and logical consistency, then they must also satisfy prior consistency. More formally:

\begin{theorem}[proven as \cref{prob_e}]\label{main-theorem}
 Let each agents $i\in\set{n}$ have credence given by probability measure $P_i$ on a set $A_i\subset X$. Suppose that $P_i|_{A_i\cap A_j} = P_j|_{A_i\cap A_j}$ for all $i,j\in\set{n}$, and that there exists a set $E\subset \bigcap_{i=1}^n A_i$ such that each $P_i(E)>0$. Then there exists a unique probability measure $P$ on $A\coloneqq \bigcup_{i=1}^n A_i$ such that $P|_{A_i} = A_i$. 
\end{theorem}

In fact, the probability measure $P$ in \cref{main-theorem} is unique, as we prove after developing the variant of sheaf theory appropriate to the measurable spaces underlying probability theory.

\section{Sheaves on Measurable Spaces}\label{sheaves}

Ordinarily, sheaves are defined on \emph{topological} spaces: given a topological space $X$ with a designated collection of ``open'' subsets, one assigns a set $\cF(U)$ to each open subset $U\subset X$, and a restriction function $\cF(U)\to\cF(V)$ whenever $V\subset U$. If, when we have a nested sequence of subsets $W\subset V\subset U$, restricting from $U$ to $V$ and then to $W$ is the same as restricting from $U$ to $W$ directly, we say that $\cF$ is a \emph{functor} or \emph{presheaf}. If in addition $\cF$ satisfies the \emph{compatibility condition}, that whenever we have a collection of open subsets $\{U_i\}_{i\in I}$ with union $U$, and elements $s_i\in \cF(U_i)$ such that the restrictions of each pair $s_i$ and $s_j$ to $U_i\cap U_j$ are equal, there must exist a unique element of $\cF(U)$ whose restriction to each $U_i$ is $s_i$, then we say $\cF$ is a \emph{sheaf}.

In this paper, we use a slightly different notion of sheaf that is suitable for a \emph{measurable space}, a set $X$ equipped with a $\sigma$-algebra of ``measurable'' subsets. Because arbitrary unions of measurable are not necessarily measurable, care must be taken to say what a cover is. We use the following definition of a cover, applicable to the collection of measurable subsets of a given measurable space:

\begin{definition}\label{cover}
A \emph{lattice} of sets $\cS$ is a collection of sets closed under finite intersections and unions. We say that a finite collection of sets $A_1,\ldots,A_n\in\cS$ with $n\ge 0$ is a \emph{cover} for another set $A\in\cS$ if $\bigcup_{i=1}^n A_i = A$.
\end{definition}

This notion of cover defines a \emph{Grothendieck pretopology} on $\cS$ (which is the reason we require closure under  finite intersections), and is therefore a suitable setting to define a sheaf. (See, for example, \cite{maclane1994} for details.) In our context, the definition of sheaf unpacks to the following.

\begin{definition}
A \emph{presheaf} $\cF$ on a lattice $\cS$ is a functor $\cF:\cS^\op\to\Set$, i.e.\ an assignment of a set $\cF(A)$ to each $A\in\cS$, together with a restriction function $\cF(A) \to \cF(B)$ whenever $A,B\in\cS$ with $B\subset A$, such that if $C\subset B\subset A$ then the composite of restrictions $\cF(A)\to \cF(B)\to\cF(C)$ equals the restriction function $\cF(A)\to\cF(C)$. 

We say that $\cF$ is a \emph{sheaf} if for each cover $\{A_i\}_{i=1}^n$ of $A$ in $\cS$, and for each choice of $s_i\in\cF(A_i)$ such that each pair $s_i$ and $s_j$ have the same restriction in $\cF(A_i\cap A_j)$, there is a unique element $s\in\cF(A)$ whose restriction to each $A_i$ is $s_i$.
\end{definition}

Our goal will be to prove that for a given measurable space $X$, we have a sheaf on $[\empty,X]$ assigning a measurable subset of $X$ to the set of all measures on it, and if we also designate a specific measurable subset $E$, we have a sheaf on $[E,X]$ assigning a measurable subset of $X$ containing $E$ to the set of probability measures on it that assign $E$ positive probability.

\begin{remark}\label{countable_counterexample}
In \cref{cover}, we do not allow covers by families of arbitrary cardinality, because, for example, a measure on $\R$ is not determined by the measure of each individual point. The reason we do not allow countably infinite covers is more subtle: the family of uniform probability measures $P_n$ on $\set{n}$ for $n\in\N=\{1,2,3,\ldots\}$ are compatible in the sense of all restricting to each other, but there is no ``uniform'' probability measure on all of $\N$ that restricts to each $P_n$.
\end{remark}

\section{Sheaves of Measures}

Fix a measurable space $X$, whose $\sigma$-algebra of measurable subsets we denote using interval notation as $[\empty,X]$, i.e.\ the collection of all measurable subsets $A$ such that $\empty\subset A\subset X$. Every measurable subset $A\subset X$ inherits the structure of a measurable space, with $\sigma$-algebra $[\empty,A]$. Given a measure $\mu$ on $X$, we denote its restriction to $A$ by $\mu|_A$, which is just the ordinary restriction of the function $\mu: [\empty,X]\to[0,\infty]$ to $[\empty,A]$.

\begin{theorem}\label{measure_sheaf}
 Define $\Meas$ to be the functor $[\empty,X]^\op\to\Set$ sending:
 \begin{itemize}
 \item each measurable set $A$ to the set $\Meas(A) \coloneqq \{$measures $\mu$ on $A\}$, and 
 \item each inclusion of measurable sets $B\subseteq A$ to the restriction function $\Meas(A)\to\Meas(B)$ sending $\mu \mapsto \mu|_B$.
 \end{itemize}
 This functor $\Meas: [\empty,X]^\op\to\Set$ is a sheaf.
\end{theorem}

\begin{proof}
 Let $A_1,\dots,A_n$ be measurable subsets of $X$, and $\mu_i$ a measure on $A_i$ for each $i\in \set{n}$. Suppose that for each $i,j\in\set{n}$, we have $\mu_i|_{A_i\cap A_j} = \mu_j|_{A_i\cap A_j}$. We must show that there exists a unique measure $\mu$ on $A \coloneqq \bigcup_{i=1}^n A_i$ such that $\mu|_{A_i} = \mu_i$ for each $i\in\set{n}$.
 
 Let $B\subset A$. Using the disjoint union decomposition $B = \coprod_{i=1}^n (B \cap A_i \setminus \bigcup_{j<i} A_j)$, if $\mu$ exists we must have
 \[\mu(B) \coloneqq \sum_{i=1}^n \mu_i\left(B \cap A_i \setminus \bigcup_{j<i} A_j\right),\]
 which we use as its definition. Each term in the sum is automatically nonnegative and countably additive, so $\mu$ is a measure, and all that remains is to check that $\mu|_{A_k} = \mu_k$ for each $k\in\set{n}$. If we assume that $B\subset A_k$, then we have
\[ \mu(B) = \sum_{i=1}^n \mu_i\left(B \cap A_i \setminus \bigcup_{j<i} A_j\right)=  \sum_{i=1}^n \mu_k\left(B \cap A_i \setminus \bigcup_{j<i} A_j\right) = \mu_k(B),\]
the middle equality holding since $\mu_i|_{A_i\cap A_k}= \mu_k|_{A_i\cap A_k}$ and the last holding since $\mu_k$ is a measure on $A_k$.
\end{proof}

\begin{corollary}\label{finite_measure_sheaf}
 The subfunctor $\FinMeas\subset \Meas: [\empty,X]^\op \to \Set$ sending each measurable subset $A\subset X$ to the set of \emph{finite} measures on $A$ is also a sheaf.
\end{corollary}

\begin{proof}
 Given compatible finite measures $\mu_i$ on $A_i$, the total measure of $A = \bigcup_{i=1}^n A_i$ is bounded by
 \[\mu(A) \le \sum_{i=1}^n \mu_i(A_i) < \infty\]
 since the sum is finite and each term is finite. (This is the \emph{only} place in the paper where we use the fact that our coverings are finite families.)
\end{proof}

\section{Sheaves of Probability}

In this section we fix a measurable space $X$ with algebra of sets $[\empty,X]$ and a measurable subset $E\subset X$.

\begin{definition}
 Given a probability measure $P$ on $A$, and a subset $B\subset A$ for which $P(B)>0$, we may define the \emph{restriction} of $P$ to $B$ by
 \[P|_B(\cdot) \coloneqq P(\cdot)/P(B).\]
 This is a probability measure on $B$; it is the one obtained by first restricting $P$ as a measure and then normalizing it so that the measure of $B$ is $1$.
\end{definition}

\begin{theorem}\label{prob_e}
Define $\Prob_E$ to be the functor $[E,X]^\op\to\Set$ sending:
\begin{itemize}
\item each measurable set $A$ containing $E$ to the set $\Prob_E(A) \coloneqq \{$probability measures $P$ on $A$ such that $P(E)>0\}$, and 
\item each inclusion of measurable sets $B\subset A$ containing $E$ to the restriction function $\Prob_E(A)\to\Prob_E(B)$ sending $P\mapsto P|_B$. 
 \end{itemize}
 This functor $\Prob_E:[E,X]^\op\to\Set$ is a sheaf.
\end{theorem}

\begin{proof}
 Let $A_1,\dots,A_n$ be a collection of measurable subsets of $X$, all containing $E$, and for each $i\in\set{n}$ let $P_i$ be a probability measure on $A_i$ such that $P_i(E)>0$. Suppose that for all $i,j\in\set{n}$ we have $P_i|_{A_i\cap A_j} = P_j|_{A_i\cap A_j}$. Let $A = \bigcup_{i=1}^n A_i$. We must show that there exists a unique probability measure $P$ on $A$ such that $P|_{A_i} = P_i$ for each $i\in\set{n}$.
 
 Existence: Scale each probability measure $P_i$ to an ordinary finite measure $\mu_i$ on $A_i$ by defining  $\mu_i(\cdot) \coloneqq P_i(\cdot) / P_i(E)$, so that $\mu_i(E) = 1$. Then it follows that $\mu_i|_{A_i\cap A_j} = \mu_j|_{A_i\cap A_j}$; indeed, if $B\subset A_i \cap A_j$, then
 \begin{align*}
  \mu_i(B) &= P_i(B)/P_i(E) \\
  &= \frac{P_i(B)/P_i(A_i\cap A_j)}{P_i(E)/P_i(A_i\cap A_j)}\\
  &= \frac{P_i|_{A_i\cap A_j}(B)}{P_i|_{A_i\cap A_j}(E)}
 \end{align*}
so since $P_i|_{A_i\cap A_j} = P_j|_{A_i\cap A_j}$ we must have $\mu_i(B) = \mu_j(B)$. 

Therefore by \cref{finite_measure_sheaf} there must exist a finite measure $\mu$ on $A$ such that $\mu|_{A_i} = \mu_i$ for each $i\in\set{n}$. Define a probability measure $P$ on $A$ by $P(\cdot)\coloneqq \mu(\cdot)/\mu(A)$; this is the desired gluing of the $P_i$. Indeed, if $B\subset A_i$, then 
\[ P|_{A_i}(B) = \frac{P(B)}{P(A_i)} = \frac{\mu(B)/\mu(A)}{\mu(A_i)/\mu(A)} = \frac{\mu_i(B)}{\mu_i(A_i)} = \frac{P_i(B)/P_i(E)}{P_i(A_i)/P_i(E)} = P_i(B).\]
since $P_i(A_i) = 1$. Therefore $P|_{A_i} = P_i$ as desired.

Uniqueness: Conversely, suppose that $P'$ is another probability measure on $A$ such that $P'_{A_i} = P_i$ for each $i\in\set{n}$. Scaling $P'$ to a measure $\mu'$ such that $\mu'(E) = 1$, we again find that $\mu'|_{A_i} = \mu_i = \mu|_{A_i}$ and therefore $\mu' = \mu$ by \cref{measure_sheaf}. It follows that $P' = P$ upon dividing through by $\mu'(A) = \mu(A)$. 
\end{proof}

\section{Conclusion and Further Directions}

The main results of this paper can be generalized in various directions. \Cref{measure_sheaf,finite_measure_sheaf,prob_e} remain valid when we consider measures that are merely \emph{finitely additive}, rather than the usual countably additive measures, since our finite covers do not require us to break any sets into countable disjoint unions. This may be useful in contexts in which countable additivity leads to contradictions; see \cite{Ross2012}.

On the other hand, if we do wish to allow for countable covering families, then \cref{measure_sheaf} still holds, while \cref{finite_measure_sheaf} holds if we change ``finite'' to ``$\sigma$-finite.'' It may be possible to generalize \cref{prob_e} to this setting if we work not with probability measures, but with equivalence classes of $\sigma$-finite measures up to scaling: the family of probability measures from \cref{countable_counterexample} would glue to become the equivalence class of the ordinary counting measure on $\N$, and similarly the uniform probability measures on $[-n,n]$ would glue into the class of Lebesgue measure on $\R$.

A further generalization might be to pass from sheaves on lattices of sets to sheaves on more general categories. It may be possible to create a framework for consistent beliefs among agents with self-locating uncertainty, for whom the same or indistinguishable agents may find themselves embedded in a scenario in more than one way (say, at different times or after duplication).

\bibliographystyle{plain} 
\bibliography{RefList}

\end{document}